\documentclass{amsart}
\usepackage{graphicx,amssymb,amsmath,amsfonts,amscd,hyperref,youngtab,textcomp}
\newtheorem{thm}{Theorem}[section]

\newtheorem{lem}[thm]{Lemma}
\newtheorem{prop}[thm]{Proposition}
\newtheorem{defn}[thm]{Definition}

\newcommand{\Hom}{\mathrm{Hom}}

\newcommand{\QQ}{\bar{\mathbb Q}_\ell}
\newcommand{\R}{\mathrm R}

\newcommand{\AAA}{\mathbb A}
\newcommand{\PP}{\mathbb P}
\newcommand{\FF}{\mathcal F}
\newcommand{\GG}{\mathbb G}
\newcommand{\GGG}{\mathcal G}

\newcommand{\HH}{\mathrm H}
\newcommand{\HHH}{{\mathcal H}}
\newcommand{\LL}{{\mathcal L}}

\newcommand{\Dbc}{{\mathcal D}^b_c}

\newcommand{\Swan}{\mathrm{Swan}}

\newcommand{\FT}{\mathrm{FT}}
\newcommand{\RR}{\mathcal R}
\newcommand{\Gmk}{\GG_{m,\bar k}}

\begin{document}

% \title{Katz-Radon transform of $\ell$-adic representations}
% \author{Antonio Rojas-Le\'on\affil{1}}
% \address{\affilnum{1}{Departamanto de \'Algebra,
% Universidad de Sevilla, Apdo 1160, 41080 Sevilla, Spain. \\ E-mail: arojas@us.es}}
% \thanks{Partially supported by P08-FQM-03894 (Junta de Andaluc\'{\i}a), MTM2010-19298 and FEDER}

%\correspdetails{arojas@us.es}

\title{Katz-Radon transform of $\ell$-adic representations}
\author{Antonio Rojas-Le\'on}
\address{Departamanto de \'Algebra,
Universidad de Sevilla, Apdo 1160, 41080 Sevilla, Spain. \\ E-mail: arojas@us.es}
\thanks{Partially supported by P08-FQM-03894 (Junta de Andaluc\'{\i}a), MTM2010-19298 and FEDER}
\renewcommand{\thefootnote}{}

\begin{abstract}
We prove a simple explicit formula for the local Katz-Radon transform of an $\ell$-adic representation of the Galois group of the fraction field of a strictly henselian discrete valuation ring with positive residual characteristic, which can be defined as the local additive convolution with a fixed tame character. The formula is similar to one proved by D. Arinkin in the $\mathcal D$-module setting, and answers a question posed by N. Katz.
\end{abstract}

\maketitle

\section{Introduction}

In \cite[3.4.1]{katz1996rls}, N. Katz defines some functors on the category of continuous $\ell$-adic representations of the inertia groups $I_0$ and $I_\infty$ of the projective line over $\bar k$ at $0$ and infinity, where $\bar k$ is the algebraic closure of a finite field of characteristic $p$ and $\ell$ is a prime different from $p$. These functors arise during his study of middle convolution of sheaves on the affine line and, roughly speaking, correspond to locally convolving a representation with a fixed tame character $\LL_\chi$ of $I_0$ or $I_\infty$. They are defined using G. Laumon's local Fourier transform functors, and in fact correspond to taking the tensor product with the conjugate tame character $\LL_{\bar\chi}$ on the other side of the equivalence of categories given by these functors. Katz asks \cite[3.4.1]{katz1996rls} whether there is a simple expression for the functors defined in this way.

Recently, D. Arinkin \cite{arinkin} has studied the analog of Katz's functor in $\mathcal D$-module theory: if $K$ is a field of characteristic $0$, $K((x))$ is the field of Laurent series over $K$ and $\mathcal D_x$ the ring of differential operators with coefficients in $K((x))$, the local Katz-Radon transform for a given $\lambda\in K-{\mathbb Z}$ is an equivalence of categories $\rho_\lambda:{\mathcal D}_x$-mod$\to{\mathcal D}_x$-mod, originally defined in \cite{d2002radon}. Arinkin proves the simple formula \cite[Theorem C]{arinkin}
$$
\rho_\lambda(\FF)\cong\FF\otimes{\mathcal K}^{\lambda(a+1)}
$$
for any $\FF\in{\mathcal D}_x$-mod with a single slope $a$, where ${\mathcal K}^{\mu}$ is the Kummer ${\mathcal D}_x$-module of rank $1$ generated by $\mathbf e$, on which the derivative acts by
$$
\frac{d}{dx}{\mathbf e}=\frac{\mu}{x}{\mathbf e}.
$$

In this article we will prove a similar formula in the $\ell$-adic case. More precisely, for a fixed tame $\ell$-adic character $\LL_\chi$ and an $\ell$-adic representation $\FF$ of $I_0$, let
$$
\rho_\chi(\FF):=\FT^{\psi,-1}_{(0,\infty)}(\LL_{\bar\chi}\otimes\FT^\psi_{(0,\infty)}\FF)
$$
where $\FT^\psi_{(0,\infty)}$ denotes Laumon's local Fourier transform functor. If $\FF$ has a single slope $a=c/d$ (with $c,d$ relatively prime positive integers), we will prove that there is an isomorphism of $I_0$-representations
$$
\rho_\chi(\FF)\cong\FF\otimes\LL_\chi^{\otimes(a+1)}
$$
where $\LL_\chi^{\otimes(a+1)}$ is any $d$-th root of the character $\LL_\chi^{\otimes(c+d)}$.

 For a large class of representations $\FF$ of $I_0$ (in particular for many of those who appear in applications), the isomorphism can be proven via the explicit formulas for the local Fourier transforms given by L. Fu \cite{fu2010calculation} and A. Abbes and T. Saito \cite{abbes2010local}. In this article we take a different approach that works for any $\FF$, and is independent of any explicit expression for the local Fourier transforms.

\section{The Katz-Radon transform}

Fix a finite field $k$ of characteristic $p>0$ and an algebraic closure $\bar k$. Let $\PP^1_{\bar k}$ be the projective line over $\bar k$ and, for every $t\in\PP^1({\bar k})=\bar k\cup\{\infty\}$, denote by $I_t$ its inertia group at $t$: for $t\neq\infty$, if $x-t$ denotes a local coordinate at $t$, it is the Galois group of the fraction field of the henselization of the local ring $\bar k[x]_{(x-t)}$. We have an exact sequence \cite[1.0]{katz1988gauss}
$$
0\to P_t\to I_t\to \prod_{\ell\neq p}{\mathbb Z}_\ell(1)\to 0
$$
for every $t\in\PP^1({\bar k})$, where $P_t$ is the only $p$-Sylow subgroup of $I_t$. Moreover, there is a canonical filtration of $I_t$ by the higher ramification groups
$$
I_t^{(r)}\supseteq I_t^{(s)} \text{ for }0\leq r<s\in{\mathbb R}
$$
which are normal in $I_t$.

Fix a prime $\ell\neq p$, and denote by $\RR_t$ the abelian category of continuous $\ell$-adic representations of $I_t$ (i.e. continuous representations $\FF:I_t\to\mathrm{GL}_n(\QQ)$, whose image is in $\mathrm{GL}_n(E_\lambda)$ for some finite extension $E_\lambda$ of ${\mathbb Q}_\ell$). For every irreducible $\FF\in\RR_t$, the \emph{slope} of $\FF$ is $\inf\{r\geq 0|\FF_{|I_t^{(r)}}\text{ is trivial}\}$. It is a non-negative rational number. In general, the slopes of $\FF$ are the slopes of the irreducible components of $\FF$. For every $\FF$ there is a canonical direct sum decomposition \cite[Lemma 1.8]{katz1988gauss}
\begin{equation}\label{directsum}
\FF\cong\bigoplus_{r\geq 0}\FF^r
\end{equation}
with $\FF^r$ having a single slope $r$. The slope $0$ (tame) part will be denoted by $\FF^t$. $\FF$ is said to be \emph{tame} (respectively \emph{totally wild}) if $\FF=\FF^t$ (resp. $\FF^t=0$).

For every $r\geq 0$ let $\RR_t^r$ denote the full subcategory of $\RR_t$ consisting of representations with a single slope $r$. We have a decomposition
$$
\RR_t=\bigoplus_{r\geq 0}\RR_t^r
$$
in the sense that every $\FF\in\RR_t$ has a decomposition \eqref{directsum} and $\Hom_{\RR_t}(\FF,\GGG)=0$ if $\FF\in\RR_t^r$, $\GGG\in\RR_t^s$ and $r\neq s$ \cite[Proposition 1.1]{katz1988gauss}.

Let $k'\subseteq \bar k$ be a finite extension of $k$, and $\chi:k'^\ast\to\QQ^\ast$ a multiplicative character. By \cite[1.4-1.8]{deligne569application} there is an associated smooth Kummer sheaf $\LL_\chi$ on $\Gmk$, which is a tame character of $I_0$ (and of $I_\infty$) of the same order as $\chi$. If $k'\subseteq k''$ is another extension, the sheaves defined by $\chi$ and $\chi\circ\mathrm{Nm}_{k''/k'}:k''^\ast\to\QQ^\ast$ are isomorphic. Moreover, every tame character of $I_0$ (and of $I_\infty$) can be obtained in this way. Whenever we speak about a tame character of $I_0$, we will implicitly assume that we have made a choice of such a finite extension of $k$ and of a character.

Fix a non-trivial additive character $\psi:k\to\QQ^\ast$. The local Fourier transform functors, defined by G. Laumon in \cite{laumon1987transformation}, give equivalences of categories
$$
\FT_{(0,\infty)}^\psi:\RR_0\to\RR_\infty^{<1},
$$
$$
\FT_{(\infty,\infty)}^\psi:\RR_\infty^{>1}\to\RR_\infty^{>1}
$$
and
$$
\FT_{(\infty,0)}^\psi:\RR_\infty^{<1}\to\RR_0
$$
(where $\RR_\infty^{<1}=\bigoplus_{r<1}\RR_\infty^r$ and $\RR_\infty^{>1}=\bigoplus_{r>1}\RR_\infty^r$) that describe the relationship between the local monodromies of an $\ell$-adic sheaf on $\AAA^1_{\bar k}$ and its Fourier transform with respect to $\psi$. The Katz-Radon transform is defined in terms of them.

\begin{defn}
Fix a tame character $\LL_\chi$ of $I_0$. The (local) \emph{Katz-Radon transform} (with respect to $\LL_\chi$) is the functor $\rho_\chi:\RR_0\to\RR_0$ given by
$$
\rho_\chi(\FF)=\FT_{(0,\infty)}^{\psi,-1}(\FT_{(0,\infty)}^\psi\LL_\chi\otimes\FT_{(0,\infty)}^\psi\FF)=\FT_{(0,\infty)}^{\psi,-1}(\LL_{\bar\chi}\otimes\FT_{(0,\infty)}^\psi\FF).
$$
\end{defn}

The Katz-Radon transform is an auto-equivalence of the category $\RR_0$ (since it is a composition of three equivalences of categories). It preserves dimensions and slopes, and for tame $\FF$ it is given by $\rho_\chi(\FF)=\FF\otimes\LL_\chi$ \cite[3.4.1]{katz1996rls}. For totally wild $\FF$, it can be interpreted as the ``local additive convolution'' of $\FF$ and $\LL_\chi$ \cite[3.4.3]{katz1996rls}: if we extend $\FF$ to a smooth sheaf on $\Gmk$, tamely ramified at infinity, then $\rho_\chi(\FF)$ is the wild part of the local monodromy at $0$ of $\FF\ast\LL_\chi$, where
$$
\FF\ast\LL_\chi=\R^1\sigma_!(\FF\boxtimes\LL_\chi)
$$
and $\sigma:\AAA^2_{\bar k}\to\AAA^1_{\bar k}$ denotes the addition map (in \cite{katz1996rls}, the ``middle convolution'' is used instead, but that one differs from the one used here only by Artin-Shreier components, which are smooth at $0$ and therefore do not affect the local monodromy). Notice that, in particular, $\rho_\chi$ is independent of the choice of the additive character $\psi$.

More intrinsically, it can be described in terms of vanishing cycles functors \cite[2.7.2]{laumon1987transformation}: If $X={\mathbb A}^2_{(0,0)}$ (respectively $S=\AAA^1_{(0)}$) denotes the henselization of $\AAA^2_{\bar k}$ at $(0,0)$ (resp. the henselization of $\AAA^1_{\bar k}$ at $0$) then $\rho_\chi(\FF)\cong\R^1\Phi(\sigma,\FF\boxtimes\LL_\chi)_{(0,0)}$, where $\R\Phi(\sigma,\FF\boxtimes\LL_\chi)$ is the vanishing cycles complex for the addition map $\sigma:X\to S$ with respect to the sheaf $\FF\boxtimes\LL_\chi$ on $X$. 

Similarly, it also has an interpretation as a ``local multiplicative convolution'' \cite[Corollary 5.6]{rl2011}: If $X={\GG}^2_{m,(1,1)}$ (respectively $S={\GG}_{m,(1)}$) denotes the henselization of $\Gmk$ at $(1,1)$ (resp. the henselization of $\Gmk$ at $1$) then $\rho_\chi(\FF)\cong\R^1\Phi(\mu,\FF\boxtimes\LL_\chi)_{(1,1)}$, where $\R\Phi(\mu,\FF\boxtimes\LL_\chi)$ is the vanishing cycles complex for the multiplication map $\mu:X\to S$ with respect to the sheaf $\FF\boxtimes\LL_\chi$ on $X$, and $\FF$ and $\LL_\chi$ are viewed as representations of $I_1$ via the isomorphism $I_0\cong I_1$ that maps the uniformizer $x$ at $0$ to the uniformizer $x-1$ at $1$.

The main result of this article is the following simple expression for $\rho_\chi$:
\begin{thm}\label{main}
Let $\FF\in\RR_0$ be totally wild with a single slope $a>0$. Write $a=c/d$, where $c$ and $d$ are relatively prime positive integers. Let $\LL_\eta$ be any tame character of $I_0$ such that $\LL_\eta^{\otimes d}=\LL_{\chi}^{\otimes(c+d)}$. Then
$$
\rho_\chi(\FF)\cong\FF\otimes\LL_\eta.
$$ 
\end{thm}
In other words, we have the formula 
\begin{equation}
 \rho_\chi(\FF)\cong\FF\otimes\LL_{\chi}^{\otimes(a+1)}
\end{equation}
where $\LL_{\chi}^{\otimes(a+1)}$ stands for ``any character that can reasonably be called $\LL_{\chi}^{\otimes(a+1)}$''.

By the decomposition $\RR_0=\bigoplus_{r\geq 0}\RR_0^r$, this determines $\rho_\chi(\FF)$ for any $\FF\in\RR_0$, thus answering the question posed by N. Katz in \cite[3.4.1]{katz1996rls}.

A question that remains open is the following: in the article we prove that $\rho_\chi(\FF)\cong\FF\otimes\LL_{\eta}$, independently for any $\FF$ with slope $a$. So the functors $\RR_0^a\to\RR_0^a$ given by $\rho_\chi$ and $(-)\otimes\LL_\eta$ map any $\FF$ to isomorphic objects. Is there an actual isomorphism of functors between them? In the affirmative case, is there a simple way to construct it?

\section{Proof of the main theorem}

In this section we will prove theorem \ref{main}. We will start with the case where $\FF\in\RR_0$ is irreducible.

\begin{lem}\label{chartame}
 Let $\FF\in\RR_0$. Then $\FF^t\neq 0$ if and only if there exists $\epsilon>0$ such that for every $\GGG\in\RR_0$ with a single slope $b\in(0,\epsilon)$ we have 
$$
\Swan(\FF\otimes\GGG)>\Swan(\FF)\dim(\GGG).
$$
\end{lem}

\begin{proof}
 Suppose that $\FF^t\neq 0$, and let $a_0=0<a_1<\cdots<a_r$ be the slopes of $\FF$, with multiplicities $n_0,n_1,\ldots,n_r$. Then $\Swan(\FF)=\sum n_ia_i$. Let $\epsilon=a_1$. Then for every $\GGG\in\RR_0$ with a single slope $b\in(0,\epsilon)$ the tensor product $\FF\otimes\GGG$ has slopes $b<a_1<\cdots<a_r$ with multiplicities $n_0m,n_1m,\ldots,n_rm$ where $m=\dim(\GGG)$ by \cite[Lemma 1.3]{katz1988gauss}. Therefore
$$
\Swan(\FF\otimes\GGG)=n_0mb+\sum_{i=1}^rn_ima_i>\sum_{i=1}^rn_ima_i=\Swan(\FF)\dim(\GGG).
$$
Conversely, suppose that $\FF^t=0$, and let $a_1<\cdots<a_r$ be the slopes of $\FF$. Then for every $\GGG\in\RR_0$ with a single slope $b\in(0,a_1)$ the tensor product $\FF\otimes\GGG$ has the same slopes as $\FF$ by \cite[Lemma 1.3]{katz1988gauss}, and in particular $\Swan(\FF\otimes\GGG)=\Swan(\FF)\dim(\GGG)$. This proves the lemma, since for every $\epsilon>0$ there exist representations in $\RR_0$ with slope $b\in(0,\epsilon)$ (for instance, one may take $[n]_\ast{\mathcal H}$, where ${\mathcal H}\in\RR_0$ has slope $a>0$ and $n$ is a prime to $p$ integer greater than $a/\epsilon$ \cite[1.13.2]{katz1988gauss}).
\end{proof}

For any two objects $K,L\in\Dbc(\AAA^1_{\bar k},\QQ)$, we will denote by $K\ast L\in\Dbc(\AAA^1_{\bar k},\QQ)$ their additive convolution:
$$
K\ast L=\R\sigma_!(K\boxtimes L)
$$
where $\sigma:\AAA^2_{\bar k}\to\AAA^1_{\bar k}$ is the addition map.

\begin{lem}\label{associative}
 Let $K,L,M\in\Dbc(\AAA^1_{\bar k},\QQ)$. Then 
$$
\R\Gamma_c(\AAA^1_{\bar k},(K\ast L)\otimes M)\cong\R\Gamma_c(\AAA^1_{\bar k},K\otimes((\tau_{-1}^\ast L)\ast M))
$$
where $\tau_{-1}:\AAA^1_{\bar k}\to\AAA^1_{\bar k}$ is the additive inversion.
\end{lem}

\begin{proof}
 We have
\begin{align*}
&\R\Gamma_c(\AAA^1_{\bar k},(K\ast L)\otimes M)=\R\Gamma_c(\AAA^1_{\bar k},\R\sigma_!(K\boxtimes L)\otimes M)=
\\
&=\R\Gamma_c(\AAA^1_{\bar k},\R\sigma_!((K\boxtimes L)\otimes\sigma^\ast M))
=\R\Gamma_c(\AAA^2_{\bar k},(K\boxtimes L)\otimes\sigma^\ast M)
\end{align*}
by the projection formula. If $\pi_1,\pi_2:\AAA^2_{\bar k}\to\AAA^1_{\bar k}$ are the projections then 
$$
\R\Gamma_c(\AAA^2_{\bar k},(K\boxtimes L)\otimes\sigma^\ast M)
=\R\Gamma_c(\AAA^2_{\bar k},\pi_1^\ast K\otimes \pi_2^\ast L\otimes\sigma^\ast M).
$$
Consider the automorphism $\phi:\AAA^2_{\bar k}\to\AAA^2_{\bar k}$ given by $(x,y)\mapsto (x+y,-y)$. Then $\sigma=\pi_1\circ\phi$, $\pi_1=\sigma\circ\phi$ and $\tau_{-1}\circ\pi_2=\pi_2\circ\phi$. It follows that
\begin{align*}
&\R\Gamma_c(\AAA^2_{\bar k},\pi_1^\ast K\otimes \pi_2^\ast L\otimes\sigma^\ast M)\cong\R\Gamma_c(\AAA^2_{\bar k},\phi^\ast\pi_1^\ast K\otimes \phi^\ast\pi_2^\ast L\otimes\phi^\ast\sigma^\ast M)=
\\
&=\R\Gamma_c(\AAA^2_{\bar k},\sigma^\ast K\otimes \pi_2^\ast \tau_{-1}^\ast L\otimes\pi_1^\ast M)=\R\Gamma_c(\AAA^1_{\bar k},\R\sigma_!(\sigma^\ast K\otimes \pi_2^\ast \tau_{-1}^\ast L\otimes\pi_1^\ast M))\cong
\\
&\cong\R\Gamma_c(\AAA^1_{\bar k},K\otimes\R\sigma_!((\tau_{-1}^\ast L)\boxtimes M))=\R\Gamma_c(\AAA^1_{\bar k},K\otimes((\tau_{-1}^\ast L)\ast M)).
\end{align*}
\end{proof}

If $\FF$ is a smooth $\QQ$-sheaf on $\Gmk$ which is totally wild at $0$, then for every $t\in\bar k$ the sheaf $\FF\otimes\LL_{\chi(t-x)}$ (extended by zero to $\AAA^1_{\bar k}$) is totally wild at $0$ and has no punctual sections (where $\LL_{\chi(t-x)}$ is the pull-back of the Kummer sheaf $\LL_\chi$ under the map $x\mapsto t-x$), so its only non-zero cohomology group with compact support is $\HH^1_c$. We conclude that the only non-zero cohomology sheaf of $\FF[0]\ast\LL_\chi[0]\in\Dbc(\AAA^1_{\bar k},\QQ)$ is $\HHH^1=\R^1\sigma_!(\FF\otimes\LL_\chi)$. We will denote this sheaf by $\FF\ast\LL_\chi$.

\begin{lem}\label{sameswan}
 Let $\FF,\GGG\in\RR_0$. Then 
$$
\Swan(\rho_\chi(\FF)\otimes\GGG)=\Swan(\FF\otimes\GGG).
$$
\end{lem}

\begin{proof}
 By additivity of the Swan conductor, we may assume that $\FF$ is irreducible, and in particular that it has a single slope $a\geq 0$. If $a=0$ then $\rho_\chi(\FF)\cong\FF\otimes\LL_\chi$, so the equality is clear. Suppose that $a>0$. By \cite[Theorem 1.5.6]{katz1986local}, $\FF$ and $\GGG$ can be extended to smooth sheaves on $\Gmk$, tamely ramified at infinity, which we will also denote by $\FF$ and $\GGG$. Let $\FF$ and $\GGG$ be also their extensions by zero to $\AAA^1_{\bar k}$.

Using the compatibility between Fourier transform with respect to $\psi$ and convolution \cite[Proposition 1.2.2.7]{laumon1987transformation}, we have
$$
\FF\ast\LL_\chi=\FT^{\bar\psi}(\FT^\psi\FF\otimes\FT^\psi\LL_\chi)=\FT^{\bar\psi}(\FT^\psi\FF\otimes\LL_{\bar\chi}),
$$
where $\FT^\psi\FF$ denotes the ``naive'' Fourier transform in the sense of \cite[8.2]{katz1988gauss}, that is, the $(-1)$-th cohomology sheaf of the Fourier transform of $\FF[1]\in\Dbc(\AAA^1_{\bar k},\QQ)$ (which is its only non-zero cohomology sheaf, since $\FF$ is totally wild at zero and therefore it is Fourier \cite[Lemma 8.3.1]{katz1988gauss}).

Let $n$ be the rank of $\FF$, and denote by $\FF_{(\infty)}\in\RR_\infty$ its local monodromy at infinity, which is a tame representation of $I_\infty$. By Ogg-Shafarevic \cite[Expos\'e X, Corollaire 7.12]{grothendieck1977cohomologie}, $\FT^\psi\FF$ is smooth on $\Gmk$ of rank $na+n=n(a+1)$. By Laumon's local Fourier transform theory \cite[Theorem 13]{katz1988travaux}, $\FT^\psi\FF$ has a single slope $\frac{a}{a+1}$ at infinity, with multiplicity $n(a+1)$, and its monodromy at $0$ has a trivial part of dimension $na$ and its quotient is the dual $\widehat{\FF_{(\infty)}}$ of $\FF_{(\infty)}$. Then $\FT^\psi\FF\otimes\LL_{\bar\chi}$ also has a single slope $\frac{a}{a+1}$ at infinity with multiplicity $n(a+1)$, and its monodromy $\mathcal M$ at $0$ sits in an exact sequence
\begin{equation}\label{eq1}
0\to\LL_{\bar\chi}^{\oplus na}\to{\mathcal M}\to\widehat{\FF_{(\infty)}}\otimes\LL_{\bar\chi}\to 0.
\end{equation}
 Its inverse Fourier transform, by Ogg-Shafarevic, is smooth of rank $n(a+1)$ on $\Gmk$, and by local Fourier transform its wild part at $0$ has slope $a$ with multiplicity $n$. In fact, this wild part is simply $\rho_\chi(\FF)$ by the additive convolution interpretation of $\rho_\chi$. Its monodromy at infinity sits in an exact sequence
\begin{equation}\label{eq2}
0\to\LL_{\chi}^{\oplus na}\to(\FF\ast\LL_\chi)_{(\infty)}\to{\FF_{(\infty)}}\otimes\LL_{\chi}\to 0
\end{equation}
obtained from \eqref{eq1} by local Fourier transform.

So $\FF\ast\LL_\chi$ has rank $n(a+1)$ on $\Gmk$, and its monodromy at $0$ is the direct sum of $\rho_\chi(\FF)$ and a constant part of dimension $na=\Swan(\FF)$. So
$$
\Swan_0((\FF\ast\LL_\chi)\otimes\GGG)=\Swan(\rho_\chi(\FF)\otimes\GGG)+\Swan(\FF)\Swan(\GGG).
$$
In particular, by Ogg-Shafarevic, the Euler characteristic of the sheaf $(\FF\ast\LL_\chi)\otimes\GGG$ (extended by zero to $\AAA^1_{\bar k}$) is $-\Swan(\rho_\chi(\FF)\otimes\GGG)-\Swan(\FF)\Swan(\GGG)$. Using lemma \ref{associative}, proper base change, and the fact that $\chi(\Gmk,K\otimes\LL_\chi)=\chi(\Gmk,K)$ for any object $K\in\Dbc(\Gmk,\QQ)$, we get
\begin{align*}
&\Swan(\rho_\chi(\FF)\otimes\GGG)+\Swan(\FF)\Swan(\GGG)=\chi(\AAA^1_{\bar k},(\FF[1]\ast\LL_\chi[1])\otimes\GGG)= \\
& =\chi(\AAA^1_{\bar k},\LL_\chi\otimes(\tau_{-1}^\ast\FF[1]\ast\GGG[1]))=\chi(\Gmk,\tau_{-1}^\ast\FF[1]\ast\GGG[1])= \\
& =\chi(\AAA^1_{\bar k},\tau_{-1}^\ast\FF[1]\ast\GGG[1])-\mathrm{rank}_0(\tau_{-1}^\ast\FF[1]\ast\GGG[1])= \\
& =\chi(\AAA^1_{\bar k},\FF[1])\chi(\AAA^1_{\bar k},\GGG[1])-\chi(\AAA^1_{\bar k},\FF[1]\otimes\GGG[1])= \\
& =\Swan(\FF)\Swan(\GGG)+\Swan(\FF\otimes\GGG)
\end{align*}
where $\mathrm{rank}_0$ of a derived category object denotes the alternating sum of the ranks at $0$ of its cohomology sheaves, so $\Swan(\rho_\chi(\FF)\otimes\GGG)=\Swan(\FF\otimes\GGG)$.
\end{proof}
  
\begin{prop}\label{irred1}
 Let $\FF\in\RR_0$ be totally wild and irreducible. Then there exists a tame character $\LL_\eta$ of $I_0$ such that $\rho_\chi(\FF)\cong\FF\otimes\LL_\eta$.
\end{prop}

\begin{proof}
 Let $\widehat\FF$ be the dual representation. We claim that the tame part of $\rho_\chi(\FF)\otimes\widehat\FF$ is non-zero. By lemma \ref{chartame}, it suffices to show that there is an $\epsilon>0$ such that, for any $\GGG\in\RR_0$ with slope $b\in(0,\epsilon)$, $\Swan(\rho_\chi(\FF)\otimes\widehat\FF\otimes\GGG)>\Swan(\rho_\chi(\FF)\otimes\widehat\FF)\dim(\GGG)$. But by lemma \ref{sameswan}, we have
$$
 \Swan(\rho_\chi(\FF)\otimes\widehat\FF\otimes\GGG)=\Swan(\FF\otimes\widehat\FF\otimes\GGG)
$$
and
$$
\Swan(\rho_\chi(\FF)\otimes\widehat\FF)=\Swan(\FF\otimes\widehat\FF)
$$
and, since $\widehat\FF$ is the dual of $\FF$, the tensor product $\FF\otimes\widehat\FF$ has a trivial quotient and, in particular, has non-trivial tame part. By lemma \ref{chartame}, there exists $\epsilon>0$ such that, for any $\GGG\in\RR_0$ with slope $b\in(0,\epsilon)$, $\Swan(\FF\otimes\widehat\FF\otimes\GGG)>\Swan(\FF\otimes\widehat\FF)\dim(\GGG)$.

Since the tame part of $\rho_\chi(\FF)\otimes\widehat\FF$ is non-zero and it is a direct summand, it contains a tame character $\LL_{\eta}$ of $I_0$ as a subrepresentation. Then $$\rho_\chi(\FF)\otimes\widehat\FF\otimes\LL_{\bar\eta}=\rho_\chi(\FF)\otimes\widehat{\FF\otimes\LL_{\eta}}=\Hom(\FF\otimes\LL_{\eta},\rho_\chi(\FF))$$ contains a trivial subrepresentation, so $\Hom_{I_0}(\FF\otimes\LL_\eta,\rho_\chi(\FF))\neq 0$.
Since both $\rho_\chi(\FF)$ and $\FF\otimes\LL_\eta$ are irreducible, any non-zero $I_0$-equivariant map $\FF\otimes\LL_\eta\to\rho_\chi(\FF)$ must be an isomorphism.
\end{proof}

\begin{prop}\label{irred2}
Let $\FF\in\RR_0$ be totally wild and irreducible of dimension $n$ and slope $a$, and let $\LL_\eta$ be a tame character of $I_0$ such that $\rho_\chi(\FF)\cong\FF\otimes\LL_\eta$. Then $\LL_\eta^{\otimes n}\cong\LL_\chi^{\otimes n(a+1)}$.
\end{prop}

\begin{proof}
 Extend $\FF$ to a smooth $\ell$-adic sheaf on $\Gmk$, tamely ramified at infinity, also denoted by $\FF$. Let $\FF$ also denote its extension by zero to $\AAA^1_{\bar k}$. By the proof of lemma \ref{sameswan}, the sheaf $\FF\ast\LL_\chi$ is smooth on $\Gmk$, its monodromy at $0$ is the direct sum of $\rho_\chi(\FF)\cong\FF\otimes\LL_\eta$ and a trivial part of dimension $na$, and its monodromy at infinity sits in the exact sequence \eqref{eq2}. Its determinant is then a smooth sheaf of rank $1$ on $\Gmk$, whose monodromy at $0$ is $\det(\FF)\otimes\LL_\eta^{\otimes n}$, and whose monodromy at $\infty$ is $\det(\FF_{(\infty)})\otimes\LL_\chi^{\otimes n(a+1)}$.

Then $\widehat{\det(\FF)}\otimes\LL_{\bar\eta}^{\otimes n}\otimes\det(\FF\ast\LL_\chi)$ is a rank 1 smooth sheaf on $\Gmk$, with trivial monodromy at $0$ and tamely ramified at infinity. Since the tame fundamental group of $\AAA^1_{\bar k}$ is trivial, we conclude that
$$
\det(\FF\ast\LL_\chi)\cong\det(\FF)\otimes\LL_{\eta}^{\otimes n}
$$
as sheaves on $\Gmk$. Comparing their monodromies at infinity gives the desired isomorphism.
\end{proof}

It remains to show that any such $\LL_\eta$ works.

\begin{lem}\label{induced}
 Let $\FF\in\RR_0$ be irreducible of dimension $n$, and let $\LL_\eta$ be a tame character of $I_0$ such that $\LL_\eta^{\otimes n}$ is trivial. Then $\FF\otimes\LL_\eta\cong\FF$.
\end{lem}

\begin{proof}
 Write $n=n_0p^\alpha$, where $\alpha\geq 0$ and $n_0$ is prime to $p$. Since the $p$-th power operation permutes the tame characters of $I_0$ preserving their order, $\LL_\eta^{\otimes n_0}$ must be the trivial character. Now by \cite[1.14.2]{katz1988gauss}, $\FF$ is induced from a $p^\alpha$-dimensional representation $\GGG$ of $I_0(n_0)$, the unique open subgroup of $I_0$ of index $n_0$. Then
$$
\FF\otimes\LL_\eta=(\mathrm{Ind}^{I_0}_{I_0(n_0)}\GGG)\otimes\LL_\eta\cong\mathrm{Ind}^{I_0}_{I_0(n_0)}(\GGG\otimes\mathrm{Res}^{I_0}_{I_0(n_0)}\LL_\eta)=\mathrm{Ind}^{I_0}_{I_0(n_0)}(\GGG)=\FF
$$
since the restriction of $\LL_\eta$ to $I_0(n_0)$ is trivial.
\end{proof}

We can now finish the proof of theorem \ref{main} for irreducible representations

\begin{prop}
 Let $\FF\in\RR_0$ be irreducible of slope $a>0$. Write $a=c/d$, where $c$ and $d$ are relatively prime positive integers. Let $\LL_\eta$ be any tame character of $I_0$ such that $\LL_\eta^{\otimes d}=\LL_\chi^{\otimes(c+d)}$. Then
$$
\rho_\chi(\FF)\cong\FF\otimes\LL_\eta.
$$
\end{prop}

\begin{proof}
 Let $n$ be the dimension of $\FF$. By propositions \ref{irred1} and \ref{irred2}, there exists a tame character $\LL_{\eta'}$ of $I_0$ such that $\rho_\chi(\FF)\cong\FF\otimes\LL_{\eta'}$, and $\LL_{\eta'}^{\otimes n}\cong\LL_\chi^{\otimes n(a+1)}$. Since the Swan conductor $na=nc/d$ of $\FF$ is an integer, $n$ must be divisible by $d$. Then
\begin{align*}
(\LL_{\bar\eta'}\otimes\LL_{\eta})^{\otimes n}&=\LL_{\bar\eta'}^{\otimes n}\otimes\LL_{\eta}^{\otimes d(n/d)}=\\ &\LL_{\bar\chi}^{\otimes n(a+1)}\otimes\LL_{\chi}^{\otimes(c+d)n/d}=\LL_{\bar\chi}^{\otimes n(a+1)}\otimes\LL_{\chi}^{\otimes n(a+1)}={\mathbf 1}
\end{align*}
so, by lemma \ref{induced},
$$
\rho_\chi(\FF)\cong\FF\otimes\LL_{\eta'}\cong(\FF\otimes\LL_{\eta'})\otimes(\LL_{\bar\eta'}\otimes\LL_{\eta})=\FF\otimes\LL_\eta.
$$
\end{proof}

\begin{proof}[{\bf Proof of theorem \ref{main}}]
 The functors $\RR^a_0\to\RR^a_0$ given by $\FF\mapsto\rho_\chi(\FF)$ and $\FF\mapsto\FF\otimes\LL_\eta$ are equivalences of categories, so they preserve direct sums. It is enough then to prove the isomorphism for indecomposable representations.

So let $\FF\in\RR_0^a$ be indecomposable of length $m$. Then by \cite[Lemma 3.1.6, Lemma 3.1.7(3)]{katz1996rls} there exist an irreducible $\FF_0\in\RR_0^a$ and a (necessarily tame) indecomposable unipotent ${\mathcal U}_m\in\RR_0$ of dimension $m$ such that $\FF=\FF_0\otimes{\mathcal U}_m$. Since $\FF$ is a succesive extension of $m$ copies of $\FF_0$, by exactness $\rho_\chi(\FF)$ is a succesive extension of $m$ copies of $\rho_\chi(\FF_0)\cong\FF_0\otimes\LL_\eta$, which is irreducible. By \cite[Lemma 3.1.7(2)]{katz1996rls}, there is a unipotent $\mathcal U\in\RR_0$ of dimension $m$ such that $\rho_\chi(\FF)\cong\FF_0\otimes\LL_\eta\otimes{\mathcal U}$.

Since $\rho_\chi$ is an equivalence of categories, $\rho_\chi(\FF)$ must be indecomposable, so $\mathcal U$ itself must be indecomposable. Therefore ${\mathcal U}\cong{\mathcal U}_m$ and
$$
\rho_\chi(\FF)\cong\FF_0\otimes\LL_\eta\otimes{\mathcal U}_m\cong\FF\otimes\LL_\eta.
$$ 
\end{proof}

\section{Some variants}

We will consider now representations of the inertia group $I_\infty$ at infinity. For any $\FF\in\RR_\infty$ of slope $>1$, we can take its local Fourier transform $\FT^\psi_{(\infty,\infty)}\FF$, which is again in the same category. In \cite[3.4.4]{katz1996rls}, N. Katz asks about a simple formula for
$$
\rho_\chi'(\FF):=\FT_{(\infty,\infty)}^{\psi,-1}(\LL_{\bar\chi}\otimes\FT^\psi_{(\infty,\infty)}\FF),
$$
which is an auto-equivalence of the category of continuous $\ell$-adic representations of $\RR_\infty$ with slopes $>1$. It can be interpreted as the wild part of the monodromy at infinity of the (additive) convolution $\FF\ast\LL_\chi$ \cite[3.4.6]{katz1996rls}, where $\FF$ is any extension of the representation $\FF$ to a smooth sheaf on $\Gmk$ tamely ramified at $0$. In this section we will prove

\begin{thm}\label{main2}
 Let $\FF\in\RR_\infty$ be totally wild with a single slope $a>1$. Write $a=c/d$, where $c$ and $d$ are relatively prime positive integers. Let $\LL_\eta$ be any tame character of $I_\infty$ such that $\LL_\eta^{\otimes d}=\LL_{\bar\chi}^{\otimes(c-d)}$. Then
$$
\rho'_\chi(\FF)\cong\FF\otimes\LL_\eta.
$$ 
\end{thm}

In other words, we have the formula 
\begin{equation}
 \rho'_\chi(\FF)\cong\FF\otimes\LL_{\bar\chi}^{\otimes(a-1)}
\end{equation}
where $\LL_{\bar\chi}^{\otimes(a-1)}$ stands for ``any character that can reasonably be called $\LL_{\bar\chi}^{\otimes(a-1)}$''.

The proof is very similar to the one for $\rho_\chi$. Since every representation in $\RR_\infty$ is a direct sum of representations with single slopes, we can assume that $\FF$ has a single slope $a$.

\begin{lem}\label{sameswan2}
 Let $\FF,\GGG\in\RR_\infty$ be totally wild, with $\FF$ having all slopes $>1$. Then 
$$
\Swan(\rho'_\chi(\FF)\otimes\GGG)=\Swan(\FF\otimes\GGG).
$$
\end{lem}

\begin{proof}
 We can assume that $\FF$ has a single slope $a>1$. Extend $\FF$ and $\GGG$ to smooth sheaves on $\Gmk$, tamely ramified at $0$, which we will also denote by $\FF$ and $\GGG$ (as well as their extensions by zero to $\AAA^1_{\bar k}$).

Let $n$ be the rank of $\FF$, and denote by $\FF_{(0)}$ its local monodromy at $0$, which is a tame representation of $I_0$. Since all slopes of $\FF$ at infinity are $>1$, it is a Fourier sheaf \cite[Lemma 8.3.1]{katz1988gauss}, so its Fourier transform is a single sheaf that we will denote by $\FT^\psi\FF$. By Ogg-Shafarevic, $\FT^\psi\FF$ is smooth on $\Gmk$ of rank $na$. By Laumon's local Fourier transform theory \cite[Remark 9]{katz1988travaux}, it has a single positive slope $\frac{a}{a-1}$ at infinity with multiplicity $n(a-1)$ and tame part isomorphic to $\widehat{\FF_{(0)}}$, and it is unramified at $0$. Then $\FT^\psi\FF\otimes\LL_{\bar\chi}$ also has a single slope $\frac{a}{a-1}$ at infinity with multiplicity $n(a-1)$, tame part isomorphic to $\LL_{\bar\chi}\otimes\widehat{\FF_{(0)}}$, and its monodromy at $0$ is a direct sum of $na$ copies of $\LL_{\bar\chi}$. 

 Its inverse Fourier transform, by Ogg-Shafarevic, is smooth of rank $n(a-1)\frac{a}{a-1}+n=n(a+1)$ on $\Gmk$, and by local Fourier transform its monodromy at infinity is the direct sum of $\rho'_\chi(\FF)$ and $na=\Swan(\FF)$ copies of $\LL_\chi$. At $0$ is has trivial part of rank $na$, whith quotient isomorphic to $\LL_\chi\otimes\FF_{(0)}$. So
$$
\Swan_\infty((\FF\ast\LL_\chi)\otimes\GGG)=\Swan(\rho'_\chi(\FF)\otimes\GGG)+\Swan(\FF)\Swan(\GGG).
$$
We conclude exactly as in lemma \ref{sameswan}.
\end{proof}

Using lemma \ref{chartame} as in proposition \ref{irred1} we deduce

\begin{prop}\label{irred3}
 Let $\FF\in\RR_\infty$ be irreducible with slope $>1$. Then there exists a tame character $\LL_\eta$ of $I_\infty$ such that $\rho'_\chi(\FF)\cong\FF\otimes\LL_\eta$.
\end{prop}

\begin{prop}\label{irred4}
Let $\FF\in\RR_\infty$ be irreducible of dimension $n$ and slope $a>1$, and let $\LL_\eta$ be a tame character of $I_\infty$ such that $\rho'_\chi(\FF)\cong\FF\otimes\LL_\eta$. Then $\LL_\eta^{\otimes n}\cong\LL_{\bar\chi}^{\otimes n(a-1)}$.
\end{prop}

\begin{proof}
 Extend $\FF$ to a smooth $\ell$-adic sheaf on $\Gmk$, tamely ramified at $0$, also denoted by $\FF$, and let $\FF$ also denote its extension by zero to $\AAA^1_{\bar k}$. By the proof of lemma \ref{sameswan2}, the sheaf $\FF\ast\LL_\chi$ is smooth on $\Gmk$, its monodromy at infinity is the direct sum of $\rho'_\chi(\FF)\cong\FF\otimes\LL_\eta$ and $na$ copies of $\LL_\chi$, and its monodromy at 0 has trivial part of dimension $na$ with quotient isomorphic to $\LL_\chi\otimes\FF_{(0)}$. Its determinant is then a smooth sheaf of rank $1$ on $\Gmk$, whose monodromy at $\infty$ is $\det(\FF)\otimes\LL_\eta^{\otimes n}\otimes\LL_\chi^{\otimes na}$, and whose monodromy at $0$ is $\det(\FF_{(0)})\otimes\LL_\chi^{\otimes n}$.

We conclude, as in proposition \ref{irred2}, that
$$
\det(\FF\ast\LL_\chi)\cong\det(\FF)\otimes\LL_{\eta}^{\otimes n}\otimes\LL_\chi^{\otimes na}
$$
as sheaves on $\Gmk$. Comparing their monodromies at $0$ gives the desired isomorphism.
\end{proof}

The remainder of the proof of theorem \ref{main2} is identical to the one for $\rho_\chi$.

\bigskip

We have a third variant, for representations $\FF\in\RR_\infty$ with slopes $<1$:
$$
\rho_\chi''(\FF):=\FT_{(\infty,0)}^{\psi,-1}(\LL_{\bar\chi}\otimes\FT^\psi_{(\infty,0)}(\FF)),
$$
which is again an auto-equivalence of the category of continuous $\ell$-adic representations of $\RR_\infty$ with slopes $<1$. As in the $\rho_\chi$ case we have $\rho_\chi''(\FF)\cong\FF\otimes\LL_\chi$ for $\FF$ tame. The corresponding formula for wild $\FF$ is
\begin{thm}
 Let $\FF\in\RR_\infty$ be totally wild with a single slope $a<1$. Write $a=c/d$, where $c$ and $d$ are relatively prime positive integers. Let $\LL_\eta$ be any tame character of $I_\infty$ such that $\LL_\eta^{\otimes d}=\LL_{\chi}^{\otimes(d-c)}$. Then
$$
\rho''_\chi(\FF)\cong\FF\otimes\LL_\eta.
$$
\end{thm}

\begin{proof}
 Let $\GGG:=\FT^\psi_{(\infty,0)}(\FF)\in\RR_0$, which has slope $\frac{a}{1-a}=\frac{c}{d-c}$ \cite[Theorem 13]{katz1988travaux}. The statement is then equivalent to
$$
 \FT_{(\infty,0)}^{\psi,-1}(\LL_{\bar\chi}\otimes\GGG)\cong\LL_\eta\otimes\FT_{(\infty,0)}^{\psi,-1}(\GGG)
$$
or
$$
\FT^\psi_{(\infty,0)}(\LL_\eta\otimes\FT_{(\infty,0)}^{\psi,-1}(\GGG))\cong\GGG\otimes\LL_{\bar\chi}.
$$
But the left hand side is just $\rho_{\bar\eta}(\GGG)$, since the inverse of $\FT^\psi_{(\infty,0)}$ is $\FT^{\bar\psi}_{(0,\infty)}$ with respect to the conjugate additive character, and $\rho_\chi$ does not depend on the choice of the non-trivial additive character $\psi$. So the isomorphism follows from theorem \ref{main}.
\end{proof}

\section*{Funding}

This work was partially supported by P08-FQM-03894 (Junta de Andaluc\'{\i}a), MTM2010-19298 and FEDER

% \bibliographystyle{amsalpha}
% \bibliography{katz-radon}

\bibliographystyle{plainnat}

\end{document}